\documentclass{article}
\usepackage{sectsty}
\usepackage{amssymb}
\usepackage{amsfonts}
\usepackage{amsthm}
\usepackage{fancyhdr}
\usepackage{indentfirst}
\usepackage{lscape}
\usepackage{amsmath}
\usepackage{graphicx}
\input xy
\xyoption{all}
\allsectionsfont{\centering \normalsize}

\newtheorem{theorem}{Theorem}[section]
\newtheorem{definition}[theorem]{Definition}
\newtheorem{conjecture}[theorem]{Conjecture}
\newtheorem{proposition}[theorem]{Proposition}
\newtheorem{lemma}[theorem]{Lemma}
\newtheorem{corollary}[theorem]{Corollary}
\begin{document}
\title{Brou\'e's Abelian Defect Group Conjecture for the Tits Group}
\author{Daniel Robbins}
\date{}
\maketitle
\begin{center}
\textbf{Abstract}\\
\end{center}
\noindent In this paper we prove that Brou\'e's abelian defect group conjecture holds for the Tits group $^2F_4(2)'$. Also we prove that under certain conditions we are able to lift derived equivalences and use this to prove Brou\'e's conjecture for the group $^2F_4(2)$.
\section{Introduction}
\indent Much of modern representation theory is concerned with the relationship between representations of a group and those of its subgroups. Brou\'e's abelian defect group conjecture attempts to understand this relationship using derived equivalences and can be stated as follows:
\begin{conjecture}[Brou\'e's Abelian Defect Group Conjecture] \label{a1}  Let $k$ be an algebraically closed field of prime characteristic $p$ and $G$ be a finite group. If $A$ is a block, of the group algebra $kG$, with an abelian\footnote[2]{The abelian condition is not well understood but it is essential as the principal block of $kSz(8)$ in characteristic $2$ has a non-abelian defect group and is not derived equivalent to it's Brauer correspondent. However, there are still some interesting cases where ``Brou\'e's Conjecture'' holds and the defect group is not abelian.}  defect group $D$, and $B$ is the Brauer correspondent of $A$, a block of the group algebra $kN_G (D)$, then $D^b$(mod-$A$) and $D^b$(mod-$B$) are equivalent as triangulated categories.   
\end{conjecture}
This conjecture is already known in many cases, for example when the defect group is cyclic, see \cite{rr98}, or $A$ is a principal block with defect group $C_3\times C_3$, see \cite{kkw02a}.\\
\indent In many cases, where Brou\'e's conjecture is still unknown, there is known to be a stable equivalence of Morita type between $A$ and $B$, which can be seen as a consequence of the derived equivalence as the stable module category is a canonical quotient of the derived category, see \cite{jr89a}. This and an idea of Okuyama's \cite{to97} led Rickard to prove the following in his paper \cite{jr2002}:
\begin{theorem}[Rickard's theorem] \label{a2} Let $A$ and $B$ be finite-dimensional symmetric algebras,
\begin{displaymath}
F:\textrm{mod-}A\to \textrm{mod-}B
\end{displaymath}
be an exact functor inducing a stable equivalence of Morita type and let $S_{1},...,S_{n}$ be a set of representatives for the isomorphism classes of simple $A$-modules.\\
\indent If there are objects $X_{1},...,X_{n}$ of $D^{b}$(mod-$B$) such that $X_{i}$ is stably isomorphic to $F(S_{i})$, for each $1 \leq i \leq n$, and such that the following are satisfied:
\begin{enumerate}
	\item[(a)] $\textrm{Hom}_{D^b(B)}(X_i,X_j[m])=0\indent \textrm{ for m $<$ 0}$
	\item[(b)] $\textrm{Hom}_{D^b(B)}((X_i,X_j)= \left\{ \begin{array}{ll}
0 & \textrm{if $i \not = j$}\\
k & \textrm{if $i = j$}
\end{array} \right.$
	\item[(c)] $X_1,...,X_n \textrm{ generate } D^b(\textrm{mod-}B) \textrm{ as a triangulated category}$
	\end{enumerate}
then $D$(Mod-$A$) and $D$(Mod-$B$) are equivalent as triangulated categories. $\square$
\end{theorem}
In \S 3 we use this theorem to prove that Brou\'e's conjecture holds for the Tits group.\\
\indent Following Al-Nofayee \cite{a} we make the following definition concerning the objects in Rickard's theorem:
\begin{definition} \label{a3} Given an algebra $A$, we say the objects $X_1,..., X_n$ of $D^b(\textrm{mod-}A)$ are a cohomologically schurian set of generators if they satisfy the following conditions:
\begin{enumerate}
	\item[(a)] $\textnormal{Hom}_{D^b(A)}(X_i,X_j[m])=0\ \textrm{for m $<$ 0}$
	\item[(b)] $\textnormal{Hom}_{D^b(A)}(X_i,X_j)= \left\{ \begin{array}{ll}
0 & \textrm{if $i \not = j$}\\
k & \textrm{if $i = j$}
\end{array} \right.$
	\item[(c)] $X_1,...,X_n \textrm{ generate } D^b(\textrm{mod-}A) \textrm{ as a triangulated category}$
\end{enumerate}
\end{definition}
\indent If we take $k$ to be a group of characteristic $p$, $\widetilde{G}$ to be a group with a normal subgroup $G$, $\widetilde{N}$ to be a group with a normal subgroup $N$, $H=\widetilde{G}/G\cong \widetilde{N}/N$, $A$ to be a block of $kG$ and $B$ to be a block of $kN$ both stable under the action of $H$, set $\widetilde{A}=A\uparrow ^{\widetilde{G}}$ and $\widetilde{B}=B\uparrow ^{\widetilde{N}}$, note that $\widetilde{A}$ and $\widetilde{B}$ are blocks of $k\widetilde{G}$ and $k\widetilde{N}$ respectively, and list the simple $A$-modules as $s_1,...,s_m,s_{11},...,s_{1k_1},s_{21},...,s_{2k_2},...,s_{r1},\\
,...,s_{rk_r}$ such that $H$ fixes $s_1,...,s_m$ and for each $1\leq i\leq r$ permutes $s_{i1},...,s_{ik_i}$, then in \S 4 we prove that:
\begin{theorem} \label{a4} Let $\widetilde{F}:\overline{mod}-\widetilde{A}\rightarrow \overline{mod}-\widetilde{B}$ and $F:\overline{mod}-A\rightarrow \overline{mod}-B$ are stable equivalences of Morita type such that 
\begin{equation}
F(-)\uparrow ^{\widetilde{N}}\cong \widetilde{F}(-\uparrow ^{\widetilde{G}}),
\end{equation}
and $X_1,...,X_m,X_{11},...,X_{1k_1},...,X_{r1},...,X_{rk_r}\in D^b(\textnormal{mod}-B)$ is a cohomologically schurian set of generators such that $H$ permutes $X_{i1},...,X_{ik_i}$ for $1\leq i\leq r$ and fixes $X_1,...,X_m$, $X_{ij}\cong F(s_{ij})$ for $1\leq i\leq r$ and $1\leq j\leq k_r$, and $X_i\cong F(s_i)$ for $1\leq i\leq m$. Then $\widetilde{A}$ and $\widetilde{B}$ are derived equivalent. $\square$
\end{theorem}
\indent Finally, in \S 5, we combine the results of \S 3 and \S 4 to prove that Brou\'e's conjecture holds for the group $^2F_4(2)$ and in the appendix we give the Loewy layers of the modules required in \S 3 and \S 5.\\
\indent When we refer to programmes we are refering to the computer package MAGMA \cite{magma}.
\section{Notation}
\indent Throughout this paper we use the following notation. Let $A$ be a ring, we denote by $1_A$, Z($A$) and $A^{\times}$ the unit element, the centre and the set of units of $A$ respectively. The category of left $A$-modules is denoted by mod-$A$ and the category of right modules is denoted $A$-mod, unless otherwise stated all $A$-modules are assumed to be left $A$-modules. Also, we denote the category of projective $A$-modules by Proj-$A$ and the category of finitely-generated projective $A$-modules by $P_A$.\\
\indent We always take $G$ to be a group, $p$ to be a prime number and ($\mathcal{O},K,k$) to be a $p$-modular splitting system for all subgroups of $G$, that is, $\mathcal{O}$ is a complete discrete valuation ring of rank one with quotient field $K$ of characteristic $0$ and residue field $k=\mathcal{O}/$rad($\mathcal{O}$) of characteristic $p$ such that $K$ and $k$ are splitting fields for all subgroups of $G$. When we talk of an $\mathcal{O}G$-module we mean a finitely generated right $\mathcal{O}G$-module which is free as an $\mathcal{O}$-module. Given an $\mathcal{O}G$-module $M$ we have a $kG$-module $M/$[$M$.rad($\mathcal{O}$)], we often abuse notation and also denote this by $M$.\\
\indent Our complexes will be \emph{cochain} complexes, so the differentials will have degree 1 and if 
\begin{displaymath}
X=\cdots \rightarrow X^{i-1}\stackrel{d^{i-1}}{\rightarrow}X^i\stackrel{d^i}{\rightarrow}X^{i+1}\rightarrow \cdots
\end{displaymath}
is a cochain complex, then $X[m]$ denotes $X$ shifted $m$ places to the left, that is $X[m]^i=X^{i+m}$ and $d^i[m]=(-1)^md^{i+m}$.\\
\indent If $A$ is an additive category then $K(A)$ will be the chain homotopy category of cochain complexes over $A$, $K^-(A)$ will be the full subcategory consisting of complexes $X$ that are bounded above, i.e. $X^i=0$ for $i\gg 0$, and $K^b(A)$ will be the full subcategory of bounded complexes, i.e. complexes $X$ with $X^i=0$ for all but finitely many $i$.\\
\indent If $A$ is an abelian category then $D(A)$ will be the derived category of cochain complexes over $A$, and $D^-(A)$ and $D^b(A)$ will be the full subcategories of complexes that are bounded above and bounded respectively.\\
\indent For a natural number $n$, $C_n$ denotes the cyclic group of order $n$ and for a subgroup $E$ of Aut($G$), $G\rtimes E$ denotes the semi-direct product.\\
\section{Brou\'e's Conjecture for the Tits Group}
\indent In this section we take $G$ to be the simple group $^2F_4(2)'$, which is the derived subgroup of $^2F_4(2)$ and is often referred to as the Tits group, its order is $17971200=2^{11}3^35^213$. We also take $k$ to be an algebraically closed field of characteristic $p$. We prove that Brou\'e's conjecture holds for $G$.
\begin{theorem} It is enough to prove the conjecture for the principal block in characteristic $p=5$.
\end{theorem}
\begin{proof} Using GAP\cite{gap} we see that non-principal blocks have trivial defect groups, so we only have to consider principal blocks. Sylow $2$-subgroups and Sylow $3$-subgroups are not abelian so there is nothing to prove. A Sylow $13$-subgroup is cyclic and Brou\'e's conjecture has been proven for cyclic defect groups in \cite{rr98}. The only remaining case to check is characteristic $5$ where we do have an abelian Sylow $5$-subgroup $C_5\times C_5$.
\end{proof}
\indent We now take $k$ to be an algebraically closed field of characteristic $5$. The principal block, $A=ekG$, of $kG$ has elementary abelian defect group $P$ of order $25$, we set $N=N_G(P)$, $B=fkG$ to be the principal block of $kN$ and we prove that Brou\'e's conjecture holds for $A$. Note that we are taking $e$ and $f$ to be the block idempotents of $A$ and $B$ respectively.\\
\indent There are $14$ simple $kG$-modules in $A$ which we shall denote by $S_1,..., S_{14}$. To prove our result we shall find $14$ complexes $X_1,..., X_{14}$ which satisfy the conditions of Rickard's theorem \ref{a2}. First we have to find a stable equivalence, in this case the equivalence is simply restriction since $C_G(Q)=C_N(Q)$ for every non-trivial subgroup $Q\leq P$. Therefore, if we take $C$ to be $ekGf$ considered as a complex concentrated in degree zero, then $C$ induces a stable equivalence between $A$ and $B$. Moreover if we take $RG_i$ to be the Green correspondant of $S_i$ then:
\begin{displaymath}
S_i\otimes C:\ \cdots \rightarrow 0\rightarrow RG_i\oplus (projs)\rightarrow \cdots
\end{displaymath}
where the non-zero term is in degree zero.\\
\indent We now want to construct complexes $X_1,...,X_{14}$ such that $X_i$ is stably isomorphic to $S_i\otimes _AC$ and which satisfy the conditions (a)-(c) of Rickard's theorem. To define the $X_i$'s we shall describe the complex $P_i$ in $K^b(P_B)$ such that we have a triangle
\begin{displaymath}
P_i\rightarrow S\otimes _AC\rightarrow X_i\rightarrow P[1]
\end{displaymath}
\indent First we found an isotypy between $A$ and $B$ and hence found that from this isotypy we have the perfect isometry:
\begin{displaymath}
\left( \begin{array}{ccc}
\chi_1 \\
\chi_{26a} \\
\chi_{26b} \\
\chi_{27a} \\
\chi_{27b} \\
\chi_{78} \\
\chi_{351a} \\
\chi_{351b} \\
\chi_{351c} \\
\chi_{624a} \\
\chi_{624b} \\
\chi_{702a} \\
\chi_{702b} \\
\chi_{1728} \\
\chi_{2048a} \\
\chi_{2048b} \\
\end{array} \right)
\longrightarrow
\left( \begin{array}{ccc}
\phi_{1a} \\
\phi_{1c} \\
\phi_{1e} \\
\phi_{2f} \\
\phi_{2e} \\
\phi_{3a} \\
-\phi_{24b} \\
\phi_{1b} \\
-\phi_{24a} \\
-\phi_{1d} \\
-\phi_{1f} \\
\phi_{2a} \\
\phi_{2d} \\
\phi_{3b} \\
-\phi_{2b} \\
-\phi_{2c} \\
\end{array} \right)
\end{displaymath}
where the subscripts denote the degree of the characters. This perfect isometry is suspected to be the map induced by the character $\chi$ of a complex $K\otimes _\mathcal{O}X$, where $X$ is a tilting complex of $e\mathcal{O}G$-$f\mathcal{O}N$-bimodules. We also found the character $\chi _C$ of the complex $C$.\\
\indent Using decomposition matrices we can calculate the map on Grothendieck groups $K_0($mod-$A)\rightarrow K_0($mod-$B)$ induced by the character $\chi$ of $k\otimes _{\mathcal{O}}X$. Assuming the tilting complex $X$ lifts our complex $C$ we have a triangle
\begin{displaymath}
Q\rightarrow C\rightarrow k\otimes _{\mathcal{O}}X\rightarrow Q[1]
\end{displaymath}
where $Q$ is a bounded complex of projective $ekG$-$kN$-bimodules. We then have a map $K_0($mod-$A)\rightarrow K_0(P_B)$ induced by $\chi -\chi _C$ and the image of the simple $A$-module $S_i$ will then be the image in $K_0(P_B)$ of the complex $P_i$ required to define $X_i$ and hence this can act as a guide in choosing the terms of $P_i$.\\
\indent We used the method outlined above to find the following complexes:
\begin{equation}
\begin{aligned}
X_1:\cdots &\rightarrow RG_1\\
X_2:\cdots &\rightarrow RG_2\\
X_3:\cdots &\rightarrow RG_3\\
X_4:\cdots &\rightarrow P_{11}\rightarrow P_{14}\rightarrow RG_4\\
X_5:\cdots &\rightarrow P_7\rightarrow P_{13}\rightarrow P_{13}\oplus P_7\rightarrow P_7\oplus P_{12}\oplus P_{10}\rightarrow P_{10}\oplus P_{12}\oplus P_2\rightarrow \\
&\rightarrow P_2\oplus P_{14}\rightarrow RG_5\\
X_6:\cdots &\rightarrow P_2\rightarrow P_2\oplus P_{14}\rightarrow P_{11}\oplus P_7\rightarrow RG_6\\
X_7:\cdots &\rightarrow P_7\rightarrow P_{12}\rightarrow P_{10}\rightarrow P_7\rightarrow P_3\rightarrow RG_7\\
X_8:\cdots &\rightarrow P_7\rightarrow P_{10}\rightarrow P_{12}\rightarrow P_7\rightarrow P_5\rightarrow RG_8\\
\end{aligned}
\nonumber
\end{equation}
\begin{equation}
\begin{aligned}
X_9:\cdots &\rightarrow P_2\rightarrow P_{14}\rightarrow P_{14}\oplus P_7\rightarrow P_9\oplus P_8\oplus P_7\rightarrow RG_9\\
X_{10}:\cdots &\rightarrow P_7\rightarrow P_7\oplus P_5\oplus P_3\rightarrow P_3\oplus P_5\oplus P_{13}\rightarrow P_{13}\oplus P_{13}\rightarrow \\
&\rightarrow P_{13}\oplus P_{12}\oplus P_{10}\rightarrow RG_{10}\\
X_{11}:\cdots &\rightarrow P_{13}\rightarrow P_7\oplus P_5\oplus P_{10}\rightarrow P_{10}\oplus P_{12}\rightarrow P_2\oplus P_{13}\rightarrow P_{5}\oplus P_{14}\rightarrow \\
&\rightarrow P_{11}\oplus P_{10}\rightarrow RG_{11}\\
X_{12}:\cdots &\rightarrow P_{13}\rightarrow P_7\oplus P_3\oplus P_{12}\rightarrow P_{10}\oplus P_{12}\rightarrow P_2\oplus P_{13}\rightarrow P_{3}\oplus P_{14}\rightarrow \\
&\rightarrow P_{11}\oplus P_{12}\rightarrow RG_{12}\\
X_{13}:\cdots &\rightarrow P_7\rightarrow P_{10}\rightarrow P_{12}\rightarrow P_7\rightarrow P_9\oplus P_5\rightarrow P_{13}\rightarrow RG_{13}\\
X_{14}:\cdots &\rightarrow P_7\rightarrow P_{12}\rightarrow P_{10}\rightarrow P_7\rightarrow P_8\oplus P_3\rightarrow P_{13}\rightarrow RG_{14}\\
\end{aligned}
\nonumber
\end{equation}
where in each case $RG_i$ is in degree zero. Note that enough information is given in the appendix to determine these complexes uniquely up to isomorphism.\\
\indent Finally, we want to prove that these $X_i$'s satisfy the conditions of Rickard's theorem and hence it will follow that Brou\'e's conjecture holds for $A$.\\
\indent First we know that the $X_i$'s are stably isomorphic to $S_i\otimes _A C$ by construction and then we prove that conditions (a) and (b) hold using the program contained in \cite{mh01}. All that is left to prove is that the $X_i$'s generate $D^b(B)$ as a triangulated category. Note, while reading the following proof one should refer to the appendix where I have given the Loewy structure of the homology of the complexes $X_i$.
\begin{lemma} Let $\mathcal{T}$ be the triangulated category generated by the $X_i$'s, then $\mathcal{T}=D^b(B)$. 
\end{lemma}
\begin{proof} The proof of this theorem consists of taking each complex $X_i$ and using the short exact sequences given in appendix B to prove that every composition factor of their homology, except one, is in $T$. It then follows that every composition factor is in $T$ and we eventually conclude that $T$ contains every simple $B$-module. The details are as follows.\\
\indent First, as $X_1$, $X_2$ and $X_3$ have homologies concentrated in degree zero, we have $H^0(X_1)=1\in \mathcal{T}$, $H^0(X_2)=4\in \mathcal{T}$ and $H^0(X_3)=6\in \mathcal{T}$.\\
\indent $H^{-1}(X_4)\cong 1\in \mathcal{T}$ so $H^{-2}(X_4)\in \mathcal{T}$ and it follows from short exact sequence (B.1), in appendix B, that $11\in \mathcal{T}$.\\
\indent $H^{-2}(X_9)\cong 1\in \mathcal{T}$ and $H^{-3}(X_9)\cong H^{-2}(X_4)\in \mathcal{T}$ so $H^{-4}(X_9)\in \mathcal{T}$ and $2\in \mathcal{T}$ by short exact sequence (B.2).\\
\indent $H^{-3}(X_6)\cong H^{-4}(X_9)\in \mathcal{T}$ so $H^{-2}(X_6)\in \mathcal{T}$ and $14\in \mathcal{T}$ by short exact sequence (B.3).\\ 
\indent $H^{-1}(X_{13})\cong 4\in \mathcal{T}$, $H^{-4}(X_{13})\cong H^{-2}(X_4)\in \mathcal{T}$ and $H^{-5}(X_{13})\cong H^{-4}(X_9)\in \mathcal{T}$, so $H^{-6}(X_{13})\in \mathcal{T}$.\\
\indent $H^{-3}(X_8)\cong H^{-2}(X_4)$, $H^{-4}(X_8)\cong H^{-3}(X_6)$ and $H^{-5}(X_8)\cong H^{-6}(X_{13})$, therefore, $H^{-2}(X_8)\in \mathcal{T}$ and thus by short exact sequence (B.5) we have $9\in \mathcal{T}$.\\
\indent $H^{-1}(X_{14})\cong 6\in \mathcal{T}$, $H^{-4}(X_{14})\cong H^{-2}(X_4)\in \mathcal{T}$ and $H^{-5}(X_{14})\cong H^{-4}(X_9)\in \mathcal{T}$, so $H^{-6}(X_{14})\in \mathcal{T}$.\\
\indent $H^{-3}(X_7)\cong H^{-2}(X_4)$, $H^{-4}(X_7)\cong H^{-3}(X_6)$ and $H^{-5}(X_7)\cong H^{-6}(X_{14})$, therefore, $H^{-2}(X_9)\in \mathcal{T}$ and thus by short exact sequence (B.4) we have $8\in \mathcal{T}$.\\
\indent All the composition factors of $H^{-3}(X_5)$, $H^{-4}(X_5)$ and $H^{-5}(X_5)$ have been shown to be in $\mathcal{T}$ and so we must have $H^{-6}(X_5)\in \mathcal{T}$, hence by short exact sequence (B.6), $7\in \mathcal{T}$.\\
\indent The composition factors of $H^{-4}(X_{12})$ and $H^{-5}(X_{12})$ have been shown to be in $\mathcal{T}$ and so $H^{-6}(X_{12})\in \mathcal{T}$, it now follows from short exact sequences (B.7)-(B.9) that $N_3\in \mathcal{T}$, $3$ occurs once as a composition factor of $N_3$ and the other composition factor $9$ is in $\mathcal{T}$, hence $3\in \mathcal{T}$.\\
\indent Similarly, the composition factors of $H^{-4}(X_{11})$ and $H^{-5}(X_{11})$ have been shown to be in $\mathcal{T}$ and so $H^{-6}(X_{11})\in \mathcal{T}$, it now follows from short exact sequences (B.10)-(B.12) that $N_6\in \mathcal{T}$, $5$ occurs once as a composition factor of $N_6$ and the other composition factor $8$ is in $\mathcal{T}$, hence $5\in \mathcal{T}$.\\
\indent The composition factors of $H^{-1}(X_{10})$, $H^{-2}(X_{10})$, $H^{-3}(X_{10})$ and $H^{-4}(X_{10})$ have been shown to be in $\mathcal{T}$ and so $H^{-5}(X_{10})$. The short exact sequences (B.13) and (B.14) now imply that $N_8\in \mathcal{T}$ and the only composition factor of $N_8$ not already shown to be in $\mathcal{T}$ is $12$ which occurs just once, therefore $12\in \mathcal{T}$.\\
\indent $H^{-5}(X_8)\in \mathcal{T}$ has $13$ as a composition factor with multiplicity $1$ and all other composition factors are in $\mathcal{T}$, so $13\in \mathcal{T}$.\\
\indent Finally, $10$ is a composition factor of $H^{-5}(X_7)$ with multiplicity one and is the only simple not already shown to be in $\mathcal{T}$, hence we must have $10\in \mathcal{T}$.\\
\indent As we have shown all simple $B$-modules are contained in $\mathcal{T}$ it follows that $\mathcal{T}=D^b(B)$.
\end{proof}
\begin{theorem} Brou\'e's conjecture holds for $A$.
\end{theorem}
\begin{proof} We have shown above that the $X_i$'s satisfy the conditions of Rickard's theorem and hence the result follows from Rickard's theorem itself.
\end{proof}
\indent There is actually a refinement of Brou\'e's conjecture, due to Rickard. I refer the reader to Rickard's paper \cite{jr96} for the definition of a splendid equivalence and state the refinement as:
\begin{conjecture}[Brou\'e/Rickard Conjecture] \label{a5}If $G$ is a finite group and $A$ is a block of $kG$ with an abelian defect group $D$ and Brauer correspondent $B$, then there is a splendid equivalence between $A$ and $B$.
\end{conjecture}
\indent In \cite{mh01} Holloway proves a variant of Rickard's theorem:
\begin{theorem} \label{a6} Suppose $C$ is a complex of $A$-$B$-bimodules that induces a splendid stable equivalence between $A$ and $B$ and let $\{ S_1,...,S_n\}$ be a set of representatives for the isomorphism classes of simple $A$-modules. If there are objects $X_1,...,X_n$ of $D^b(B)$ such that, for each $1\leq i\leq n$, $X_i$ is stably isomorphic to $S_i\otimes _AC$, and such that:
\begin{enumerate}
	\item[(a)] $\textrm{Hom}_{D^b(B)}((X_i,X_j[m])=0\indent \textrm{ for m $<$ 0}$
	\item[(b)] $\textrm{Hom}_{D^b(B)}((X_i,X_j)= \left\{ \begin{array}{ll}
0 & \textrm{if $i \not = j$}\\
k & \textrm{if $i = j$}
\end{array} \right.$
	\item[(c)] $X_1,...,X_n \textrm{ generate } D^b(\textrm{mod-}B) \textrm{ as a triangulated category}$
	\end{enumerate}
then there is a splendid tilting complex $X$ that lifts $C$ and induces a splendid equivalence between $A$ and $B$ such that, for each $0\leq i\leq r$, $X_i\cong S_i\otimes _AX$ in $D^b(B)$.\ $\square$
\end{theorem}
\indent Since restriction is a splendid stable equivalence it follows that our complexes satisfy the conditions of \ref{a6} and so we have:
\begin{corollary} \label{a7} The Brou\'e/Rickard conjecture holds for $A$. $\square$
\end{corollary}
\indent For completeness we now prove that given the information in the appendices, that is the complexes $X_i$ and their homology, then the complexes $X_i$ are independent, up to isomorphism in $K^b(B)$, of the choice of differentials.\\
\indent Suppose we have constructed two trimmed resolutions of $S_i\otimes _AC$, both consistent with the information in the appendices, up to degree $-n+1$. Moreover, suppose these two complexes are isomorphic and for simplicity we take this isomorphism to be identification.\\
\indent In constructing the degree $-n$ differential suppose we make two choices $\phi ,\phi '\in$ Hom$_B(P_i^{-n},K_i^{-n+1})$, where $K_i^{-n+1}=$ker$(d^{-n+1})$, and hence two choices for the $-n$ differential. We'll show that for our $X_i$'s we can find an isomorphism between these two new complexes.\\
\indent To do this consider the diagram
\[
\xymatrix{ 
P_i^{-n} \ar[dr]^{\phi}
& 
& P_i^{-n+1} \ar[r] \ar@{=}[dd]
& P_i^{-n+2} \ar[r] \ar@{=}[dd]
& \cdots \\
& K_i^{-n+1} \ar[ur]^{j} \ar@{=}[dd]
&  
&
& \\
P_i^{-n} \ar[dr]^{\phi '}
& 
& P_i^{-n+1} \ar[r]
& P_i^{-n+2} \ar[r]
& \cdots \\
& K_i^{-n+1} \ar[ur]^{j}
&  
&
& \\
}
\]
where $j:K_i^{-n+1}\rightarrow P_i^{-n+1}$ is the natural inclusion.\\
\indent If Hom$_B(P_i^{-n},H^{-n+1}(X_i))=0$, then $im(\phi )$ is the largest submodule of $K_i^{-n+1}$ for which there are no maps from $P_i^{-n}$ to the quotient. Therefore $im(\phi )=im(\phi ')$ and we must have an isomorphism $\psi$ such that
\[
\xymatrix{
P_i^{-n} \ar[rr]^{d^{-n}=j\phi} \ar[d]^{\psi}
&
& P_i^{-n+1} \ar[r] \ar@{=}[d]
& P_i^{-n+2} \ar[r] \ar@{=}[d]
& \cdots \\
P_i^{-n} \ar[rr]^{d'^{-n}=j\phi '}
&
& P_i^{-n+1} \ar[r]
& P_i^{-n+2} \ar[r]
& \cdots \\
}
\]
is a cochain map isomorphism.\\
\indent It is clear from the appendices that the only complex for which the above does not hold is $X_6$ ($n=3$).\\
\indent Now, fix a map $\phi \in$ Hom$_B(P_i^{-n},K_i^{-n+1})$, with $q:K^{-n+1}\rightarrow K^{-n+1}/im(\phi )=H^{-n+1}(X_i)$ the natural quotient. Suppose there is a map $\eta :P^{-n}\rightarrow K^{-n+1}$ such that $q\eta \neq 0$ and take $\phi '=\lambda \phi +\eta$. If the map $\eta$ factors as $\eta =\gamma d^{-n}=\gamma j\phi$, where $\gamma :P^{-n+1}\rightarrow K^{-n+1}$, then we can adjust our degree $-n+1$ identification to an isomorphism, $\lambda +j\gamma$, and so extend the cochain isomorphism to
\[
\xymatrix{
P_i^{-n} \ar[rr]^{d^{-n}=j\phi} \ar@{=}[d]
&
& P_i^{-n+1} \ar[r] \ar[d]^{\lambda +j\gamma}
& P_i^{-n+2} \ar[r] \ar@{=}[d]
& \cdots \\
P_i^{-n} \ar[rr]^{d'^{-n}=j\phi '}
&
& P_i^{-n+1} \ar[r]
& P_i^{-n+2} \ar[r]
& \cdots \\
}
\]
\indent This is the situation that the complex $X_6$ ($n=3$) is in and so this shows that as claimed our complexes, $X_i$, are independent, up to isomorphism in $K^b(B)$, of the choice of differentials.\\
\indent We note here that we can actually explicitly find the summands of a tilting complex for $B$ giving the derived equivalence above, moreover, we can use this to find the loewy structures of the projective $A$-modules. We refer the reader to \cite{robbins} for further details and other examples.
\section{Lifting Derived Equivalences}
\indent In this section we give a proof of \ref{a4}. We begin with a definition.
\begin{definition} \label{a8}Let $\widetilde{G}$ be a group with a normal subgroup $G$ and let $\widetilde{A}$ be a summand of $k\widetilde{G}$ and $A$ be a block of $G$. We say $\widetilde{A}$ and $A$ are \textsl{linked} if every simple $\widetilde{A}$-module is a summand of a simple $A$-module induced up to $\widetilde{G}$ and every simple $A$-module induced up to $\widetilde{G}$ is a semisimple $\widetilde{A}$-module.
\end{definition}
\indent We now show that under certain assumptions a block of $kG$ is always linked to a summand of $k\widetilde{G}$. We begin by proving a lemma which we shall need in the proof of proposition \ref{a10}.
\begin{lemma} \label{a9} Let $\widetilde{G}$ be a group with a normal subgroup $G$, $k$ be an algebraically closed field of characteristic $p$, $H=\widetilde{G}/G$ be a $p'$-group and $M$, $N$ and $S$ be $k\widetilde{G}$-modules, then the short exact sequence
\begin{displaymath}
0\rightarrow M\stackrel{\alpha}{\rightarrow}S\stackrel{\beta}{\rightarrow}N\rightarrow 0
\end{displaymath}
is split if and only if the short exact sequence
\begin{displaymath}
0\rightarrow M\downarrow _G\stackrel{\alpha _G}{\rightarrow}S\downarrow _G\stackrel{\beta _G}{\rightarrow}N\downarrow _G\rightarrow 0
\end{displaymath}
is split, where $\alpha _G$ and $\beta _G$ are the retrictions of $\alpha$ and $\beta$ to $kG$-module homomorphisms.
\end{lemma}
\begin{proof} The result follows from the fact that since the index of $G$ in $\widetilde{G}$ is prime to $p$, the restriction map:
\begin{displaymath}
Ext_{k\widetilde{G}}^1(N,M)\rightarrow Ext_{kG}^1(N\downarrow _G,M\downarrow _G)
\end{displaymath}
is injective.
\end{proof}
\begin{proposition} \label{a10} Let $\widetilde{G}$ be a group with a normal subgroup $G$, $k$ be an algebraically closed field of characteristic $p$ and $H=\widetilde{G}/G$ be a $p'$-group, then every simple $kG$-module induced up to $\widetilde{G}$ is semisimple.
\end{proposition}
\begin{proof} Let $S$ be a simple $kG$-module, then by Mackeys decomposition formula $S\uparrow ^{\widetilde{G}}\downarrow _G$ is semisimple. Suppose $S\uparrow ^{\widetilde{G}}$ is not semisimple, then there is a non-split short exact sequence $0\rightarrow M\rightarrow S\uparrow ^{\widetilde{G}}\rightarrow N\rightarrow 0$, but since $S\uparrow ^{\widetilde{G}}\downarrow _G$ is semisimple this short exact sequence becomes split when you apply the restriction functor and this contradicts lemma \ref{a9}.
\end{proof}
\begin{proposition} \label{a11} Let $\widetilde{G}$ be a group with a normal subgroup $G$, $k$ be an algebraically closed field of characteristic $p$, $H=\widetilde{G}/G$ be a $p'$-group and $A$ be a block of $kG$ which is stable under the action of $H$. Then there is a summand $\widetilde{A}$ of $k\widetilde{G}$ which is linked to $A$.
\end{proposition}
\begin{proof} Let $\widetilde{A}=A\uparrow ^{\widetilde{G}}$, then as $A$ is stable under the action of $H$ it follows that $\widetilde{A}$ is a summand of $k\widetilde{G}$ as a ($k\widetilde{G},k\widetilde{G}$)-bimodule. By \ref{a10} we know that every simple $A$-module induced up to $\widetilde{G}$ is a semisimple $\widetilde{A}$-module and clearly, every simple $\widetilde{A}$-module is a summand of a simple $A$-module induced up to $\widetilde{G}$. Therefore $\widetilde{A}$ and $A$ are linked. 
\end{proof}
\indent For the rest of this section we take $k$ to be an algebraically closed field of characteristic $p$, $\widetilde{G}$ to be a group with a normal subgroup $G$, $\widetilde{N}$ to be a group with a normal subgroup $N$ and we assume $H=\widetilde{G}/G\cong \widetilde{N}/N$ is a $p'$-group. We also take $A$ to a block of $kG$ stable under the action of $H$, $B$ to be a block of $kN$ also stable under the action of $H$ and set $\widetilde{A}=A\uparrow ^{\widetilde{G}}$ and $\widetilde{B}=B\uparrow ^{\widetilde{N}}$. By \ref{a11} we know that $\widetilde{A}$ and $A$ are linked and $\widetilde{B}$ and $B$ are linked. Moreover, we list the simple $A$-modules as $s_1,...,s_m,s_{11},...,s_{1k_1},s_{21},...,s_{2k_2},...,s_{r1},...,s_{rk_r}$ such that $H$ fixes $s_1,...,s_m$ and for each $1\leq i\leq r$ permutes $s_{i1},...,s_{ik_i}$, and we denote the simple $\widetilde{A}$-modules as $S_{11},...,S_{1k_1},...,S_{m1},...,S_{mk_m},\widetilde{S}_{11},...,\widetilde{S}_{1l_1},...,\\ \widetilde{S}_{r1},...,\widetilde{S}_{rl_r}$, such that the $S_{ij}$'s are summands of $s_i\uparrow ^{\widetilde{G}}$ and the $\widetilde{S}_{ij}$'s are summands of $s_{i1}\uparrow ^{\widetilde{G}}$. We should note here that $s_{ij}\uparrow ^{\widetilde{G}}\cong s_{i1}\uparrow ^{\widetilde{G}}$ for $1\leq i\leq r$ and $1\leq j\leq k_i$.
\begin{lemma} \label{b2} If $X_1,..., X_n\in D^b(\textnormal{mod-}B)$ generate $D^b(\textnormal{mod-}B)$, then \\
$X_1\uparrow ^{\widetilde{N}},..., X_n\uparrow ^{\widetilde{N}}\in D^b(\textnormal{mod-}\widetilde{B})$ generate $D^b(\textnormal{mod-}\widetilde{B})$ as a thick subcategory.
\end{lemma}
\begin{proof} Induction is an exact functor and so it sends triangles to triangles, therefore, since $X_1,..., X_n$ generate $D^b(\textnormal{mod-}B)$ it follows that $s\uparrow ^{\widetilde{N}}$ is in the triangulated category generated by $X_1\uparrow ^{\widetilde{N}},..., X_n\uparrow ^{\widetilde{N}}$, where $s$ is a simple $B$-module. As $\widetilde{B}$ and $B$ are linked we know that each simple $\widetilde{B}$-module is a summand of a module in the triangulated category generated by $X_1\uparrow ^{\widetilde{N}},..., X_n\uparrow ^{\widetilde{N}}$ and so $X_1\uparrow ^{\widetilde{N}},..., X_n\uparrow ^{\widetilde{N}}$ generate $D^b(\textnormal{mod-}\widetilde{B})$ as a thick subcategory.
\end{proof}
\begin{lemma} \label{a12} Let $X_1,...,X_n\in D^b($mod-$B)$ be a cohomologically schurian set of generators, stable under the action of $H$ such that
\begin{displaymath}
Hom_{D^b(B)/K^b(P_{B})}(X_i,X_i)=k
\end{displaymath}
and if $i\neq j$ then 
\begin{displaymath}
Hom_{D^b(B)/K^b(P_{B})}(X_i,X_j)=0.
\end{displaymath}
Then, for each $1\leq i,j\leq n$, the natural map
\begin{displaymath}
Hom_{D^b(\widetilde{B})}(X_i\uparrow ^{\widetilde{N}},X_j\uparrow ^{\widetilde{N}})\rightarrow Hom_{D^b(\widetilde{B})/K^b(P_{\widetilde{B}})}(X_i\uparrow ^{\widetilde{N}},X_j\uparrow ^{\widetilde{N}})
\end{displaymath}
is an isomorphism of rings.
\end{lemma}
\begin{proof} Clearly the natural map
\begin{displaymath}
Hom_{D^b(B)}(X_i,X_i)\rightarrow Hom_{D^b(B)/K^b(P_{B})}(X_i,X_i)
\end{displaymath}
is non-zero, since the identity map is mapped to the identity map, and so this map must be an isomorphism of rings since Hom$_{D^b(B)}(X_i,X_i)=k$ and Hom$_{D^b(B)/K^b(P_{B})}(X_i,X_i)=k$. Also, if $i\neq j$ then 
\begin{displaymath}
Hom_{D^b(B)}(X_i,X_j)\rightarrow Hom_{D^b(B)/K^b(P_{B})}(X_i,X_j)
\end{displaymath}
is trivially an isomorphism of rings since both sides are zero.\\
\indent Since induction and restriction are adjoints we know that the natural map
\begin{displaymath}
Hom_{D^b(\widetilde{B})}(X_i\uparrow ^{\widetilde{N}},X_j\uparrow ^{\widetilde{N}})\rightarrow Hom_{D^b(\widetilde{B})/K^b(P_{\widetilde{B}})}(X_i\uparrow ^{\widetilde{N}},X_j\uparrow ^{\widetilde{N}})
\end{displaymath}
is an isomorphism of rings if and only if the natural map
\begin{displaymath}
Hom_{D^b(B)}(X_i,X_j\uparrow ^{\widetilde{N}}\downarrow _N)\rightarrow Hom_{D^b(B)/K^b(P_{B})}(X_i,X_j\uparrow ^{\widetilde{N}}\downarrow _N)
\end{displaymath}
is an isomorphism of rings. By Mackey's theorem this is true if and only if the natural map 
\begin{displaymath}
Hom_{D^b(B)}(X_i,X_j^{h_1}\oplus \cdots \oplus X_j^{h_{|H|}})\rightarrow Hom_{D^b(B)/K^b(P_{B})}(X_i,X_j^{h_1}\oplus \cdots \oplus X_j^{h_{|H|}})
\end{displaymath}
is an isomorphism of rings, where $H=\{ h_1,...,h_{|H|}\}$ and we are denoting the image of $X_i$ under the action of $h$ by $X_i^h$, which is true if and only if the natural map
\begin{displaymath}
\bigoplus _{k=1}^{|H|}Hom_{D^b(B)}(X_i,X_j^{h_k})\rightarrow \bigoplus _{k=1}^{|H|}Hom_{D^b(B)/K^b(P_{B})}(X_i,X_j^{h_k})
\end{displaymath}
is an isomorphism of rings. This is clearly true since above we've shown it to be true for each summand.
\end{proof}
\indent We can now prove our result \ref{a4} which we state again for the reader's convenience.
\begin{theorem} \label{a13} Let $\widetilde{F}:\overline{mod}-\widetilde{A}\rightarrow \overline{mod}-\widetilde{B}$ and $F:\overline{mod}-A\rightarrow \overline{mod}-B$ be stable equivalences of Morita type such that 
\begin{equation}
F(-)\uparrow ^{\widetilde{N}}\cong \widetilde{F}(-\uparrow ^{\widetilde{G}})
\end{equation}
Suppose $X_1,...,X_m,X_{11},...,X_{1k_1},...,X_{r1},...,X_{rk_r}\in D^b(\textnormal{mod}-B)$ is a cohomologically schurian set of generators such that $H$ permutes $X_{i1},...,X_{ik_i}$ for $1\leq i\leq r$ and fixes $X_1,...,X_m$, $X_{ij}\cong F(s_{ij})$ for $1\leq i\leq r$ and $1\leq j\leq k_r$, and $X_i\cong F(s_i)$ for $1\leq i\leq m$. Then $\widetilde{A}$ and $\widetilde{B}$ are derived equivalent.
\end{theorem}
\begin{proof} Let $i,j\in \{1,...,m,11,21,...,r1\}$, then by Mackey's theorem we know that 
\begin{displaymath}
Hom_{D^b(\widetilde{B})}(X_i\uparrow ^{\widetilde{N}},X_j\uparrow ^{\widetilde{N}}[t])\cong Hom_{D^b(B)}(X_i\uparrow ^{\widetilde{N}}\downarrow _N,X_j[t])\cong 
\end{displaymath}
\begin{displaymath}
Hom_{D^b(B)}(X_i^{h_1}\oplus \cdots \oplus X_i^{h_{|H|}},X_j[t])
\end{displaymath}
where $H=\{ h_1,...,h_{|H|}\}$, hence
\begin{displaymath}
Hom_{D^b(\widetilde{B})}(X_i\uparrow ^{\widetilde{N}},X_j\uparrow ^{\widetilde{N}}[t])=0 \textnormal{ if $t<0$ or }i\neq j.
\end{displaymath}
\indent Let $t\in \{ 1,...,m,11,21,...,r1\}$, then by (4), \ref{a12} and the fact that $\widetilde{A}$ and $A$ are linked we know that 
\begin{displaymath}
End_{D^b(\widetilde{B})}(X_t\uparrow ^{\widetilde{N}})\cong End_{D^b(\widetilde{B})/K^b(P_{\widetilde{B}})}(X_t\uparrow ^{\widetilde{N}})\cong End_{D^b(\widetilde{A})/K^b(P_{\widetilde{A}})}(\widetilde{F}^{-1}(X_t\uparrow ^{\widetilde{N}}))
\end{displaymath}
\begin{displaymath}
\cong End_{D^b(\widetilde{A})/K^b(P_{\widetilde{A}})}(s_t\uparrow ^{\widetilde{G}})\cong M_{n_1}(k)\oplus \cdots \oplus M_{n_{q_t}}(k).
\end{displaymath}
Hence in $D^b(\widetilde{B})$ we have 
\begin{displaymath}
X_t\uparrow ^{\widetilde{N}}\cong \underbrace{Y_1^t\oplus \cdots \oplus Y_1^t}_{n_1\textnormal{-terms}}\oplus \underbrace{Y_2^t\oplus \cdots \oplus Y_2^t}_{n_2\textnormal{-terms}}\oplus \cdots \oplus \underbrace{Y_{q_t}^t\oplus \cdots \oplus Y_{q_t}^t}_{n_{q_t}\textnormal{-terms}}
\end{displaymath}
such that 
\begin{displaymath}
Hom_{D^b(\widetilde{B})}(Y_i^{t_1},Y_j^{t_2}[w])=\left\{ \begin{array}{ll}
k & \textrm{if }i=j,\ t_1=t_2\textrm{ and }w=0\\
0 & \textrm{otherwise.}
\end{array} \right.
\end{displaymath}
Moreover, since $X_{ij}\uparrow ^{\widetilde{N}}\cong X_{i1}\uparrow ^{\widetilde{N}}$ for $1\leq i\leq r$ and $1\leq j\leq k_i$ it follows from \ref{b2} that the set $Y_1^1,...,Y_{q_1}^1,...,Y_1^{r1},...,Y_{q_{r1}}^{r1}$ generates $D^b($mod-$\widetilde{B})$ as a thick subcategory and so this set is a cohomologically schurian set of generators.\\
\indent In the stable module category we have the following isomorphisms for $t\in \{ 1,...,m,11,21,...,r1\}$:
\begin{displaymath}
Y_1^t\oplus \cdots \oplus Y_1^t\oplus \cdots \oplus Y_{q_t}^t\oplus \cdots \oplus Y_{q_t}^t\cong F(s_t)\uparrow ^{\widetilde{N}}\cong \widetilde{F}(s_t\uparrow ^{\widetilde{G}})
\end{displaymath}
which is isomorphic to:
\begin{displaymath}
\left\{ \begin{array}{ll}
\widetilde{F}(S_{t1})\oplus \cdots \oplus \widetilde{F}(S_{t1})\oplus \cdots \oplus \widetilde{F}(S_{tk_t})\oplus \cdots \oplus \widetilde{F}(S_{tk_t}) & \textrm{if }t\in \{1,..,m\}\\
\widetilde{F}(\widetilde{S}_{t1})\oplus \cdots \oplus \widetilde{F}(\widetilde{S}_{t1})\oplus \cdots \oplus \widetilde{F}(\widetilde{S}_{tl_t})\oplus \cdots \oplus \widetilde{F}(\widetilde{S}_{tl_t}) & \textrm{if }t\in \{11,..,r1\}
\end{array} \right.
\end{displaymath}
therefore, without loss of generality, we can assume that in the stable module category we have 
\begin{displaymath}
Y_i^t\cong \widetilde{F}(S_{ti}) \textrm{ if }t\in \{1,..,m\}
\end{displaymath}
\begin{displaymath}
Y_i^t\cong \widetilde{F}(\widetilde{S}_{ti}) \textrm{ if }t\in \{11,..,r1\}.
\end{displaymath}
\indent We have now shown that the set $Y_1^1,...,Y_{q_1}^1,...,Y_1^{r1},...,Y_{q_{r1}}^{r1}\in D^b($mod-$\widetilde{B})$ satisfies the conditions of \ref{a2} and so $\widetilde{A}$ and $\widetilde{B}$ are derived equivalent.
\end{proof}
\indent We note here we could have done all of the above in the context of stable equivalences of splendid type and obtained the following.
\begin{corollary} \label{a14} Let $\widetilde{F}:\overline{mod}-\widetilde{A}\rightarrow \overline{mod}-\widetilde{B}$ and $F:\overline{mod}-A\rightarrow \overline{mod}-B$ be stable equivalences of splendid type such that 
\begin{equation}
F(-)\uparrow ^{\widetilde{N}}\cong \widetilde{F}(-\uparrow ^{\widetilde{G}})
\end{equation}
Suppose $X_1,...,X_m,X_{11},...,X_{1k_1},...,X_{r1},...,X_{rk_r}\in D^b(\textnormal{mod}-B)$ is a cohomologically schurian set of generators such that $H$ permutes $X_{i1},...,X_{ik_i}$ for $1\leq i\leq r$ and fixes $X_1,...,X_m$, $X_{ij}\cong F(s_{ij})$ for $1\leq i\leq r$ and $1\leq j\leq k_r$, and $X_i\cong F(s_i)$ for $1\leq i\leq m$. Then there is a splendid derived equivalence between $\widetilde{A}$ and $\widetilde{B}$.
\end{corollary}
\begin{proof} The proof is exactly the same as that for \ref{a13}, except we find that the the set $Y_1^1,...,Y_{q_1}^1,...,Y_1^{r1},...,Y_{q_{r1}}^{r1}\in D^b($mod-$\widetilde{B})$ satisfies the conditions of \ref{a6} instead of \ref{a2}.
\end{proof}
\indent To finish this section we show that the conditions (4) and (5) in the above theorems are not vacuous.
\begin{lemma} \label{a15} Suppose restriction is a stable equivalence between $\widetilde{A}$ and $\widetilde{B}$ and $A$ and $B$, then in the stable module category
\begin{displaymath}
-\downarrow ^G_{N}\uparrow _{N}^{\widetilde{N}}\cong -\uparrow _G^{\widetilde{G}}\downarrow ^{\widetilde{G}}_{\widetilde{N}}
\end{displaymath}
\end{lemma}
\begin{proof} In the stable module category the inverse functor of restriction is induction and so
\begin{displaymath}
-\downarrow ^G_{N}\uparrow _{N}^G\cong id_G(-)
\end{displaymath}
and 
\begin{displaymath}
-\uparrow _{\widetilde{N}}^{\widetilde{G}}\downarrow ^{\widetilde{G}}_{\widetilde{N}}\cong id_{\widetilde{N}}(-)
\end{displaymath}
Therefore we have
\begin{displaymath}
-\downarrow ^G_{N}\uparrow _{N}^{\widetilde{N}}\cong -\downarrow ^G_{N}\uparrow _{N}^{\widetilde{N}}\uparrow _{\widetilde{N}}^{\widetilde{G}}\downarrow ^{\widetilde{G}}_{\widetilde{N}}\cong -\downarrow ^G_{N}\uparrow _{N}^{\widetilde{G}}\downarrow ^{\widetilde{G}}_{\widetilde{N}}\cong  -\downarrow ^G_{N}\uparrow _{N}^G\uparrow _G^{\widetilde{G}}\downarrow ^{\widetilde{G}}_{\widetilde{N}}\cong -\uparrow _G^{\widetilde{G}}\downarrow ^{\widetilde{G}}_{\widetilde{N}}
\end{displaymath}
as required.
\end{proof}
\section{Brou\'e's Conjecture for the Group $^2F_4(2)$}
\indent In this section we combine the results of \S 3 and \S 4 to prove that Brou\'e's conjecture holds for the group $^2F_4(2)$.\\
\indent We have an action of $C_2$ on $^2F_4(2)'$ which permutes the simple modules and in fact we have $^2F_4(2)'\rtimes C_2\cong\ ^2F_4(2)$. If we use the notation of \S 3 then the principal block of $^2F_4(2)$ is linked to $A$, restriction is a stable equivalence between it and its Brauer correspondent and $S_i\downarrow _{N_{^2F_4(2)'}(P)}\uparrow ^{N_{^2F_4(2)}(P)}\cong S_i\uparrow ^{^2F_4(2)}\downarrow _{N_{^2F_4(2)}(P)}$ by \ref{a15}, moreover, $C_2$ permutes the complexes $X_i$. It is now clear that the conditions of \ref{a13} are satisfied and so Brou\'e's conjecture holds for the principal block of $^2F_4(2)$.\\
\indent In fact we can explicitly obtain the complexes satisfying the conditions of \ref{a2} since
\begin{displaymath}
X_1\uparrow ^{N_{^2F_4(2)}(P)}=X'_1\oplus X'_2\\
\end{displaymath}
\begin{displaymath}
X_2\uparrow ^{N_{^2F_4(2)}(P)}=X_3\uparrow ^{N_{^2F_4(2)}(P)}=X'_3\\
\end{displaymath}
\begin{displaymath}
X_4\uparrow ^{N_{^2F_4(2)}(P)}=X'_4\oplus X'_5\\
\end{displaymath}
\begin{displaymath}
X_5\uparrow ^{N_{^2F_4(2)}(P)}=X'_6\oplus X'_7\\
\end{displaymath}
\begin{displaymath}
X_6\uparrow ^{N_{^2F_4(2)}(P)}=X'_8\oplus X'_9\\
\end{displaymath}
\begin{displaymath}
X_7\uparrow ^{N_{^2F_4(2)}(P)}=X_8\uparrow ^{N_{^2F_4(2)}(P)}=X'_{10}\\
\end{displaymath}
\begin{displaymath}
X_9\uparrow ^{N_{^2F_4(2)}(P)}=X'_{11}\oplus X'_{12}\\
\end{displaymath}
\begin{displaymath}
X_{10}\uparrow ^{N_{^2F_4(2)}(P)}=X'_{13}\oplus X'_{14}\\
\end{displaymath}
\begin{displaymath}
X_{11}\uparrow ^{N_{^2F_4(2)}(P)}=X_{12}\uparrow ^{N_{^2F_4(2)}(P)}=X'_{15}\\
\end{displaymath}
\begin{displaymath}
X_{13}\uparrow ^{N_{^2F_4(2)}(P)}=X_{14}\uparrow ^{N_{^2F_4(2)}(P)}=X'_{16}\\
\nonumber
\end{displaymath}
and explicitly the $X'_i$'s are the complexes:
\begin{equation}
\begin{aligned}
X'_1:\cdots &\rightarrow RG_1\\
X'_2:\cdots &\rightarrow RG_2\\
X'_3:\cdots &\rightarrow RG_3\\
X'_4:\cdots &\rightarrow P_{10}\rightarrow P_{13}\rightarrow RG_4\\
X'_5:\cdots &\rightarrow P_{5}\rightarrow P_{11}\rightarrow RG_5\\
X'_6:\cdots &\rightarrow P_9\rightarrow P_{12}\rightarrow P_{12}\oplus P_8\rightarrow P_8\oplus P_{16}\rightarrow P_{16}\oplus P_4\rightarrow P_4\oplus P_{13}\rightarrow RG_6\\
X'_7:\cdots &\rightarrow P_8\rightarrow P_{14}\rightarrow P_{14}\oplus P_9\rightarrow P_9\oplus P_{16}\rightarrow P_{16}\oplus P_3\rightarrow P_3\oplus P_{11}\rightarrow RG_7\\
X'_8:\cdots &\rightarrow P_4\rightarrow P_4\oplus P_{13}\rightarrow P_{10}\oplus P_9\rightarrow RG_8\\
X'_9:\cdots &\rightarrow P_3\rightarrow P_3\oplus P_{11}\rightarrow P_{5}\oplus P_8\rightarrow RG_9\\
X'_{10}:\cdots &\rightarrow P_9\oplus P_8\rightarrow P_{16}\rightarrow P_{16}\rightarrow P_9\oplus P_8\rightarrow P_7\rightarrow RG_{10}\\
X'_{11}:\cdots &\rightarrow P_3\rightarrow P_{11}\rightarrow P_{11}\oplus P_8\rightarrow P_8\oplus P_{15}\rightarrow RG_{11}\\
X'_{12}:\cdots &\rightarrow P_4\rightarrow P_{13}\rightarrow P_{13}\oplus P_9\rightarrow P_9\oplus P_{15}\rightarrow RG_{12}\\
X'_{13}:\cdots &\rightarrow P_8\rightarrow P_7\oplus P_8\rightarrow P_7\oplus P_{14}\rightarrow P_{14}\oplus P_{12}\rightarrow P_{12}\oplus P_{16}\rightarrow RG_{13}\\
X'_{14}:\cdots &\rightarrow P_9\rightarrow P_9\oplus P_7\rightarrow P_7\oplus P_{12}\rightarrow P_{12}\oplus P_{14}\rightarrow P_{14}\oplus P_{16}\rightarrow RG_{14}\\
X'_{15}:\cdots &\rightarrow P_{12}\oplus P_{14}\rightarrow P_7\oplus P_8\oplus P_9\oplus P_{16}\rightarrow P_{16}\oplus P_{16}\rightarrow \\
&\rightarrow P_{12}\oplus P_{14}\oplus P_3\oplus P_4\rightarrow P_7\oplus P_{11}\oplus P_{13}\rightarrow P_{10}\oplus P_{16}\oplus P_5\rightarrow RG_{15}\\
X'_{16}:\cdots &\rightarrow P_9\oplus P_8\rightarrow P_{16}\rightarrow P_{16}\rightarrow P_9\oplus P_8\rightarrow P_{15}\oplus P_7\rightarrow P_{12}\oplus P_{14}\rightarrow RG_{16}\\
\end{aligned}
\nonumber
\end{equation}
\indent Using the Loewy layers of the homology of these complexes, contained in appendix B, and using Holloway's program \cite{mh01} it is easy to see that these complexes do indeed satisfy the conditions of Rickard's theorem.  
\begin{theorem} Brou\'e's conjecture holds for $^2F_4(2)$.
\end{theorem}
\begin{proof} The order of $^2F_4(2)$ is $35942400=2^{12}3^35^213$, and so we only have to consider $k^2F_4(2)$ for a field $k$ of characteristic 2,3,5 or 13. It turns out, and this can easily be checked using GAP \cite{gap}, that the only block of $k^2F_4(2)$ in any characteristic with a non-cyclic abelian defect group is the principal block of $k^2F_4(2)$ in characteristic $5$, hence this is the only case we have to prove and this is exactly what we have done above.
\end{proof}
\appendix
\section{Green Correspondents}
\indent Here we denote the Green correspondent of $S_i$ by $RG_i$ and the projective cover of $S_i$ by $P_i$.\\
\\
$\mathbf{^2F_4(2)'}$\\
\\
\indent The summands of $S_i\downarrow _N$, for the simple $^2F_4(2)'$-modules are listed below, from this the reader could recreate the ordering of the simple $N$-modules used in this paper.\\
\\
$S_1\downarrow _N=RG_1$\\
$S_2\downarrow _N=RG_2\oplus P_6$\\
$S_3\downarrow _N=RG_3\oplus P_4$\\
$S_4\downarrow _N=RG_4$\\
$S_5\downarrow _N=RG_5$\\
$S_6\downarrow _N=RG_6\oplus P_3\oplus P_5$\\
$S_7\downarrow _N=RG_7\oplus P_{13}\oplus P_1$\\
$S_8\downarrow _N=RG_8\oplus P_{13}\oplus P_1$\\
$S_9\downarrow _N=RG_9\oplus P_{14}\oplus P_7\oplus P_{10}\oplus P_{11}\oplus P_{12}$\\
$S_{10}\downarrow _N=RG_{10}\oplus P_{14}\oplus P_8\oplus P_7\oplus P_9\oplus P_{11}$\\
$S_{11}\downarrow _N=RG_{11}\oplus P_{14}\oplus P_{14}\oplus P_{13}\oplus P_9\oplus P_{12}\oplus P_2\oplus P_4\oplus P_6$\\
$S_{12}\downarrow _N=RG_{12}\oplus P_{14}\oplus P_{14}\oplus P_{13}\oplus P_8\oplus P_{10}\oplus P_2\oplus P_4\oplus P_6$\\
$S_{13}\downarrow _N=RG_{13}\oplus P_{14}\oplus P_{13}\oplus P_{13}\oplus P_7\oplus P_8\oplus P_9\oplus P_{10}\oplus P_{11}\oplus P_{12}\oplus P_3\oplus P_5$\\
$S_{14}\downarrow _N=RG_{14}\oplus P_{14}\oplus P_{13}\oplus P_{13}\oplus P_7\oplus P_8\oplus P_9\oplus P_{10}\oplus P_{11}\oplus P_{12}\oplus P_3\oplus P_5$\\
\\
\indent The Loewy structures of the Green correspondents $RG_i$ are shown below.\\
\\
\begin{displaymath}
\begin{array}{cc}
RG_1=&1,\\
\end{array}
\
\begin{array}{cc}
RG_2=&4,\\
\end{array}
\
\begin{array}{cc}
RG_3=&6,\\
\end{array}
\
\begin{array}{cc}
&14\\
&12\ 10\\
&4\ 6\ 13\\
RG_4=&8\ 7\ 9\\
&14\ 14\\
&11\\
&2\\
\end{array}
\end{displaymath}
\begin{displaymath}
\begin{array}{cc}
&2\ 14\\
&11\ 7\\
&14\\
RG_5=&10\ 12\\
&6\ 4\ 13\\
&9\ 8\\
&14\\
\end{array}
\
\begin{array}{cc}
&7\ 11\\
&1\ 13\\
&9\ 8\\
RG_6=&3\ 5\ 14\\
&12\ 11\ 10\\
&13\\
&7\\
\end{array}
\
\begin{array}{cc}
&3\\
&10\\
RG_7=&13\\
&8\\
&5\\
\end{array}
\
\begin{array}{cc}
&5\\
&12\\
RG_8=&13\\
&9\\
&3\\
\end{array}
\end{displaymath}
\begin{displaymath}
\begin{array}{cc}
&7\ 8\ 9\\
&3\ 5\ 1\ 14\ 14\\
&11\ 11\ 12\ 10\ 12\ 10\\
RG_9=&6\ 4\ 2\ 13\ 13\ 13\ 13\\
&9\ 7\ 9\ 8\ 7\ 8\ 8\ 9\\
&5\ 3\ 1\ 14\ 14\\
&11\ 10\ 12\\
&13\\
\end{array}
\
\begin{array}{cc}
&10\ 12\ 13\\
&8\ 7\ 9\\
&3\ 1\ 5\ 14\ 14\\
RG_{10}=&11\ 12\ 12\ 10\ 11\ 10\\
&4\ 2\ 6\ 13\ 13\ 13\ 13\\
&9\ 7\ 8\ 9\ 8\ 7\\
&1\ 5\ 3\ 14\ 14\\
&11\ 12\ 10\\
\end{array}
\end{displaymath}
\begin{displaymath}
\begin{array}{cc}
&11\ 10\\
&2\ 4\ 13\ 13\\
&7\ 9\ 9\ 8\ 8\ 7\\
RG_{11}=&3\ 5\ 1\ 14\ 14\ 14\\
&11\ 12\ 10\ 12\ 10\ 11\\
&2\ 6\ 13\ 13\\
&7\ 8\\
\end{array}
\
\begin{array}{cc}
&11\ 12\\
&6\ 2\ 13\ 13\\
&8\ 7\ 7\ 9\ 8\ 9\\
RG_{12}=&3\ 1\ 5\ 14\ 14\ 14\\
&10\ 11\ 11\ 12\ 12\ 10\\
&4\ 2\ 13\ 13\\
&7\ 9\\
\end{array}
\end{displaymath}
\begin{displaymath}
\begin{array}{cc}
&13\\
&7\ 8\\
RG_{13}=&1\ 14\\
&10\ 11\\
&13\\
\end{array}
\
\begin{array}{cc}
&13\\
&9\ 7\\
RG_{14}=&1\ 14\\
&12\ 11\\
&13\\
\end{array}
\end{displaymath}
\\ 
$\mathbf{^2F_4(2)}$\\
\\
\indent The summands of $S_i\downarrow _N$, for the simple $^2F_4(2)$-modules are listed below, from this the reader could recreate the ordering of the simple $N$-modules used in this paper.\\
\\
$S_1\downarrow _N=RG_1$\\
$S_2\downarrow _N=RG_2$\\
$S_3\downarrow _N=RG_3\oplus P_6$\\
$S_4\downarrow _N=RG_4$\\
$S_5\downarrow _N=RG_5$\\
$S_6\downarrow _N=RG_6$\\
$S_7\downarrow _N=RG_7$\\
$S_8\downarrow _N=RG_8\oplus P_7$\\
$S_9\downarrow _N=RG_9\oplus P_7$\\
$S_{10}\downarrow _N=RG_{10}\oplus P_{14}\oplus P_{12}\oplus P_2\oplus P_1$\\
$S_{11}\downarrow _N=RG_{11}\oplus P_{16}\oplus P_{11}\oplus P_{10}\oplus P_9$\\
$S_{12}\downarrow _N=RG_{12}\oplus P_{16}\oplus P_{13}\oplus P_8\oplus P_5$\\
$S_{13}\downarrow _N=RG_{13}\oplus P_{15}\oplus P_{11}\oplus P_{10}\oplus P_9$\\
$S_{14}\downarrow _N=RG_{14}\oplus P_{15}\oplus P_{13}\oplus P_8\oplus P_5$\\
$S_{15}\downarrow _N=RG_{15}\oplus P_{16}\oplus P_{15}\oplus P_{14}\oplus 2.P_{13}\oplus P_{12}\oplus 2.P_{11}\oplus 2.P_6\oplus P_4\oplus P_3$\\
$S_{16}\downarrow _N=RG_{16}\oplus 2.P_{16}\oplus 2.P_{15}\oplus 2.P_{14}\oplus P_{13}\oplus 2.P_{12}\oplus P_{11}\oplus P_{10}\oplus P_9\oplus P_8\oplus 2.P_7\oplus P_5$\\
\\
\indent The Loewy structures of the Green correspondents $RG_i$ are shown below.\\
\\
\begin{displaymath}
\begin{array}{cc}
&\\
&\\
&\\
RG_1=&1\\
&\\
&\\
&\\
\end{array}
\
\begin{array}{cc}
&\\
&\\
&\\
RG_2=&2\\
&\\
&\\
&\\
\end{array}
\
\begin{array}{cc}
&\\
&\\
&\\
RG_3=&6\\
&\\
&\\
&\\
\end{array}
\
\begin{array}{cc}
&13\\
&16\\
&6\ 12\\
RG_4=&9\ 15\\
&13\ 11\\
&5\\
&3\\
\end{array}
\
\begin{array}{cc}
&11\\
&16\\
&6\ 14\\
RG_5=&8\ 15\\
&13\ 11\\
&10\\
&4\\
\end{array}
\end{displaymath}
\begin{displaymath}
\begin{array}{cc}
&4\ 13\\
&8\ 10\\
&11\\
RG_6=&16\\
&6\ 14\\
&15\\
&13\\
\end{array}
\
\begin{array}{cc}
&3\ 11\\
&9\ 5\\
&13\\
RG_7=&16\\
&6\ 12\\
&15\\
&11\\
\end{array}
\
\begin{array}{cc}
&9\ 10\\
&1\ 14\\
&15\\
RG_8=&13\ 7\\
&10\ 16\\
&12\\
&9\\
\end{array}
\
\begin{array}{cc}
&5\ 8\\
&2\ 12\\
&15\\
RG_9=&11\ 7\\
&5\ 16\\
&14\\
&8\\
\end{array}
\end{displaymath}
\begin{displaymath}
\begin{array}{cc}
&\\
&7\\
&16\\
RG_{10}=&12\ 14\\
&15\\
&7\\
&\\
&\\
\end{array}
\
\begin{array}{cc}
&8\ 15\\
&2\ 7\ 13\ 11\\
&5\ 10\ 16\ 16\\
RG_{11}=&4\ 6\ 12\ 14\ 12\ 14\\
&9\ 8\ 15\ 15\ 15\\
&2\ 7\ 11\ 13\\
&5\ 16\\
&12\\
\end{array}
\
\begin{array}{cc}
&9\ 15\\
&1\ 7\ 13\ 11\\
&5\ 10\ 16\ 16\\
RG_{12}=&3\ 6\ 12\ 14\ 14\ 12\\
&9\ 8\ 15\ 15\ 15\\
&2\ 7\ 11\ 13\\
&5\ 16\\
&12\\
\end{array}
\end{displaymath}
\begin{displaymath}
\begin{array}{cc}
&12\ 16\\
&9\ 15\\
&1\ 7\ 13\ 11\\
RG_{13}=&5\ 10\ 16\ 16\\
&3\ 6\ 14\ 12\ 14\ 12\\
&9\ 8\ 15\ 15\\
&2\ 7\ 13\ 11\\
&5\ 16\\
\end{array}
\
\begin{array}{cc}
&14\ 16\\
&8\ 15\\
&2\ 7\ 13\ 11\\
RG_{14}=&5\ 10\ 16\ 16\\
&3\ 6\ 14\ 12\ 14\ 12\\
&9\ 8\ 15\ 15\\
&1\ 7\ 13\ 11\\
&10\ 16\\
\end{array}
\end{displaymath}
\begin{displaymath}
\begin{array}{cc}
&5\ 10\ 16\\
&4\ 3\ 6\ 14\ 12\ 14\ 12\\
&8\ 9\ 8\ 9\ 15\ 15\ 15\ 15\\
RG_{15}&1\ 2\ 7\ 7\ 13\ 11\ 11\ 13\ 13\ 11\\
&10\ 10\ 5\ 5\ 16\ 16\ 16\ 16\\
&4\ 3\ 6\ 12\ 14\ 12\ 14\\
&9\ 8\ 15\\
\end{array}
\
\begin{array}{cc}
&\\
&14\ 12\\
&8\ 9\ 15\\
RG_{16}&2\ 1\ 13\ 11\\
&5\ 10\ 16\\
&14\ 12\\
&\\
\end{array}
\end{displaymath}
\section{Cohomology Groups of the $X_i's$}
$\mathbf{^2F_4(2)'}$\\
\\
\indent Here we display the Loewy structures of the homology of $X_i$ for $1\leq i\leq 14$. The reader should refer to this when reading lemma 3.2. Also, if one wanted to recreate the complexes $X_i$ then this information would be sufficient to produce the maps I have used.\\
\\
\begin{displaymath}
\begin{array}{cc}
H^0(X_1)=&1,\\
\end{array}
\
\begin{array}{cc}
H^0(X_2)=&4,\\
\end{array}
\
\begin{array}{cc}
H^0(X_3)=&6,\\
\end{array}
\end{displaymath}
\begin{displaymath}
\begin{array}{cc}
&1\\
H^{-2}(X_4)=&11,\\
\end{array}
\
\begin{array}{cc}
&\\
H^{-1}(X_4)=&1,\\
\end{array}
\end{displaymath}
\begin{displaymath}
\begin{array}{cc}
&\\
&2\\
H^{-6}(X_5)=&7,\\
&\\
&\\
\end{array}
\
\begin{array}{cc}
&\\
&\\
H^{-5}(X_5)=&1\ 2,\\
&\\
&\\
\end{array}
\
\begin{array}{cc}
&6\ 1\ 4\\
&9\ 8\\
H^{-4}(X_5)=&14,\\
&11\\
&2\\
\end{array}
\
\begin{array}{cc}
&4\ 6\\
&9\ 8\\
H^{-3}(X_5)=&14,\\
&11\\
&2\\
\end{array}
\end{displaymath}
\begin{displaymath}
\begin{array}{cc}
&11\\
H^{-3}(X_6)=&2,\\
&\\
\end{array}
\
\begin{array}{cc}
&14\\
H^{-2}(X_6)=&11,\\
&2\\
\end{array}
\end{displaymath}
\begin{displaymath}
\begin{array}{cc}
&6\ 9\\
&3\ 8\\
&14\\
H^{-5}(X_7)=&10\ 11,\\
&2\ 13\\
&7\\
\end{array}
\
\begin{array}{cc}
&\\
&\\
&11\\
H^{-4}(X_7)=&2,\\
&\\
&\\
\end{array}
\
\begin{array}{cc}
&\\
&\\
&1\\
H^{-3}(X_7)=&11,\\
&\\
&\\
\end{array}
\
\begin{array}{cc}
&\\
&\\
&6\\
H^{-2}(X_7)=&8,\\
&\\
&\\
\end{array}
\end{displaymath}
\begin{displaymath}
\begin{array}{cc}
&4\ 8\\
&5\ 9\\
&14\\
H^{-5}(X_8)=&11\ 12,\\
&2\ 13\\
&7\\
\end{array}
\
\begin{array}{cc}
&\\
&\\
&11\\
H^{-4}(X_8)=&2,\\
&\\
&\\
\end{array}
\
\begin{array}{cc}
&\\
&\\
&1\\
H^{-3}(X_8)=&11,\\
&\\
&\\
\end{array}
\
\begin{array}{cc}
&\\
&\\
&4\\
H^{-2}(X_8)=&9,\\
&\\
&\\
\end{array}
\end{displaymath}
\begin{displaymath}
\begin{array}{cc}
&11\\
H^{-4}(X_9)=&2,\\
\end{array}
\
\begin{array}{cc}
&1\\
H^{-3}(X_9)=&11,\\
\end{array}
\
\begin{array}{cc}
&\\
H^{-2}(X_9)=&1,\\
\end{array}
\end{displaymath}
\begin{displaymath}
\begin{array}{cc}
&4\ 6\\
&9\ 8\\
H^{-5}(X_{10})=&14\ 14,\\
&10\ 11\ 12\\
&2\ 13\\
&7\\
\end{array}
\
\begin{array}{cc}
&\\
&\\
H^{-4}(X_{10})=&14,\\
&11\\
&2\\
&\\
\end{array}
\end{displaymath}
\begin{displaymath}
\begin{array}{cc}
&14\\
H^{-3}(X_{10})=&11,\\
&2\\
\end{array}
\
\begin{array}{cc}
&\\
H^{-2}(X_{10})=&4\ 6,\\
&\\
\end{array}
\
\begin{array}{cc}
&\\
H^{-1}(X_{10})=&4\ 6,\\
&\\
\end{array}
\end{displaymath}
\begin{displaymath}
\begin{array}{cc}
&1\ 4\\
&9\\
H^{-6}(X_{11})=&14,\\
&11\ 12\\
&13\\
\end{array}
\
\begin{array}{cc}
&4\ 1\\
&9\ 11\\
H^{-5}(X_{11})=&14,\\
&11\\
&2\\
\end{array}
\
\begin{array}{cc}
&8\\
&14\\
H^{-4}(X_{11})=&11,\\
&2\\
&\\
\end{array}
\end{displaymath}
\begin{displaymath}
\begin{array}{cc}
&1\ 6\\
&8\\
H^{-6}(X_{12})=&14,\\
&11\ 10\\
&13\\
\end{array}
\
\begin{array}{cc}
&6\ 1\\
&8\ 11\\
H^{-5}(X_{12})=&14,\\
&11\\
&2\\
\end{array}
\
\begin{array}{cc}
&9\\
&14\\
H^{-4}(X_{12})=&11,\\
&2\\
&\\
\end{array}
\end{displaymath}

\begin{displaymath}
\begin{array}{cc}
&4\ 8\\
&5\ 9\\
&14\\
H^{-6}(X_{13})=&12\ 11,\\
&2\ 13\\
&7\\
\end{array}
\
\begin{array}{cc}
&\\
&\\
&11\\
H^{-5}(X_{13})=&2,\\
&\\
&\\
\end{array}
\
\begin{array}{cc}
&\\
&\\
&1\\
H^{-4}(X_{13})=&11,\\
&\\
&\\
\end{array}
\
\begin{array}{cc}
&\\
&\\
&\\
H^{-1}(X_{13})=&4,\\
&\\
&\\
\end{array}
\end{displaymath}
\begin{displaymath}
\begin{array}{cc}
&6\ 9\\
&3\ 8\\
&14\\
H^{-6}(X_{14})=&10\ 11,\\
&2\ 13\\
&7\\
\end{array}
\
\begin{array}{cc}
&\\
&\\
&11\\
H^{-5}(X_{14})=&2,\\
&\\
&\\
\end{array}
\
\begin{array}{cc}
&\\
&\\
&1\\
H^{-4}(X_{14})=&11,\\
&\\
&\\
\end{array}
\
\begin{array}{cc}
&\\
&\\
&\\
H^{-1}(X_{14})=&6\\
&\\
&\\
\end{array}
\end{displaymath}
\indent We also have the following short exact sequences, the reader should also refer to this when reading lemma 3.2:
\renewcommand{\theequation}{B.\arabic{equation}}
\setcounter{equation}{0}
\begin{equation}
0\rightarrow 11\rightarrow H^{-2}(X_4)\rightarrow H^0(X_1)\rightarrow 0
\end{equation}
\begin{equation}
0\rightarrow 2\rightarrow H^{-4}(X_9)\rightarrow 11\rightarrow 0
\end{equation}
\begin{equation}
0\rightarrow H^{-4}(X_9)\rightarrow H^{-2}(X_6)\rightarrow 14\rightarrow 0
\end{equation}
\begin{equation}
0\rightarrow 8\rightarrow H^{-2}(X_7)\rightarrow 6\rightarrow 0
\end{equation}
\begin{equation}
0\rightarrow 9\rightarrow H^{-2}(X_8)\rightarrow 4\rightarrow 0
\end{equation}
\begin{equation}
0\rightarrow 7\rightarrow H^{-6}(X_5)\rightarrow 2\rightarrow 0
\end{equation}
\begin{equation}
0\rightarrow H^{-6}(X_5)\rightarrow H^{-5}(X_7)\rightarrow N_1\rightarrow 0
\end{equation}
\begin{equation}
0\rightarrow N_2\rightarrow H^{-6}(X_{12})\rightarrow 1\rightarrow 0
\end{equation}
\begin{equation}
0\rightarrow N_2\rightarrow N_1\rightarrow N_3\rightarrow 0
\end{equation}
\begin{equation}
0\rightarrow H^{-6}(X_5)\rightarrow H^{-5}(X_8)\rightarrow N_4\rightarrow 0
\end{equation}
\begin{equation}
0\rightarrow N_5\rightarrow H^{-6}(X_{11})\rightarrow 1\rightarrow 0
\end{equation}
\begin{equation}
0\rightarrow N_5\rightarrow N_4\rightarrow N_6\rightarrow 0
\end{equation}
\begin{equation}
0\rightarrow N_7\rightarrow H^{-5}(X_7)\rightarrow N_3\rightarrow 0
\end{equation}
\begin{equation}
0\rightarrow N_7\rightarrow H^{-5}(X_{10})\rightarrow N_8\rightarrow 0
\end{equation}
where the loewy structures of the $N_i$'s are:
\begin{displaymath}
\begin{array}{cc}
&6\ 9\\
&3\ 8\\
N_1=&14,\\
&10\ 11\\
&13\\
\end{array}
\
\begin{array}{cc}
&6\\
&8\\
N_2=&14,\\
&10\ 11\\
&13\\
\end{array}
\
\begin{array}{cc}
&\\
&\\
N_3=&9,\\
&3\\
&\\
\end{array}
\
\begin{array}{cc}
&4\ 8\\
&5\ 9\\
N_4=&14,\\
&12\ 11\\
&13\\
\end{array}
\end{displaymath}
\begin{displaymath}
\begin{array}{cc}
&4\\
&9\\
N_5=&14,\\
&12\ 11\\
&13\\
&\\
\end{array}
\
\begin{array}{cc}
&\\
&\\
N_6=&8,\\
&5\\
&\\
&\\
\end{array}
\
\begin{array}{cc}
&6\\
&8\\
N_7=&14,\\
&10\ 11\\
&2\ 13\\
&7\\
\end{array}
\
\begin{array}{cc}
&\\
&4\\
N_8=&9,\\
&14\\
&12\\
&\\
\end{array}
\end{displaymath}
\\
$\mathbf{^2F_4(2)}$\\
\\
\indent Here we display the Loewy structures of the homology of $X_i$ for $1\leq i\leq 16$. The reader should refer to this when reading \S 5. Also, if one wanted to recreate the complexes $X_i$ then this information would be sufficient to produce the maps I have used.\\
\\
\begin{displaymath}
\begin{array}{cc}
H^0(X_1)=&1\\
\end{array}
\
\begin{array}{cc}
H^0(X_2)=&2\\
\end{array}
\
\begin{array}{cc}
H^0(X_3)=&6\\
\end{array}
\end{displaymath}
\begin{displaymath}
\begin{array}{cc}
H^{-2}(X_4)=&1\\
&10\\
\end{array}
\
\begin{array}{cc}
H^{-1}(X_4)=&1\\
&\\
\end{array}
\
\begin{array}{cc}
H^{-2}(X_5)=&2\\
&5\\
\end{array}
\
\begin{array}{cc}
H^{-1}(X_5)=&2\\
&\\
\end{array}
\end{displaymath}
\begin{displaymath}
\begin{array}{cc}
&\\
&\\
H^{-6}(X_6)=&3\\
&9\\
&\\
\end{array}
\
\begin{array}{cc}
&\\
&\\
H^{-5}(X_6)=&2\ 3\\
&\\
&\\
\end{array}
\
\begin{array}{cc}
&2\ 6\\
&15\\
H^{-4}(X_6)=&13\\
&10\\
&4\\
\end{array}
\
\begin{array}{cc}
&6\\
&15\\
H^{-3}(X_6)=&13\\
&10\\
&4\\
\end{array}
\end{displaymath}
\begin{displaymath}
\begin{array}{cc}
&\\
&\\
H^{-6}(X_7)=&4\\
&8\\
&\\
\end{array}
\
\begin{array}{cc}
&\\
&\\
H^{-5}(X_7)=&1\ 4\\
&\\
&\\
\end{array}
\
\begin{array}{cc}
&1\ 6\\
&15\\
H^{-4}(X_7)=&11\\
&5\\
&3\\
\end{array}
\
\begin{array}{cc}
&6\\
&15\\
H^{-3}(X_7)=&11\\
&5\\
&3\\
\end{array}
\end{displaymath}
\begin{displaymath}
\begin{array}{cc}
&\\
H^{-3}(X_8)=&10\\
&4\\
\end{array}
\
\begin{array}{cc}
&13\\
H^{-2}(X_8)=&10\\
&4\\
\end{array}
\
\begin{array}{cc}
&\\
H^{-3}(X_9)=&5\\
&3\\
\end{array}
\
\begin{array}{cc}
&11\\
H^{-2}(X_9)=&5\\
&3\\
\end{array}
\end{displaymath}
\begin{displaymath}
\begin{array}{cc}
&6\ 15\\
&15\ 7\\
H^{-5}(X_{10})=&13\ 11\\
&5\ 10\ 16\\
&4\ 3\ 12\ 14\\
&8\ 9\\
\end{array}
\end{displaymath}
\begin{displaymath}
\begin{array}{cc}
&\\
&\\
H^{-4}(X_{10})=&5\ 10\\
&3\ 4\\
&\\
&\\
\end{array}
\
\begin{array}{cc}
&\\
&\\
H^{-3}(X_{10})=&1\ 2\\
&5\ 10\\
&\\
&\\
\end{array}
\
\begin{array}{cc}
&\\
&\\
H^{-2}(X_{10})=&6\\
&15\\
&\\
&\\
\end{array}
\end{displaymath}
\begin{displaymath}
\begin{array}{cc}
H^{-4}(X_{11})=&5\\
&3\\
\end{array}
\
\begin{array}{cc}
H^{-3}(X_{11})=&2\\
&5\\
\end{array}
\
\begin{array}{cc}
H^{-2}(X_{11})=&2\\
&\\
\end{array}
\end{displaymath}
\begin{displaymath}
\begin{array}{cc}
H^{-4}(X_{12})=&10\\
&4\\
\end{array}
\
\begin{array}{cc}
H^{-3}(X_{12})=&1\\
&10\\
\end{array}
\
\begin{array}{cc}
H^{-2}(X_{12})=&1\\
&\\
\end{array}
\end{displaymath}
\begin{displaymath}
\begin{array}{cc}
&6\\
&15\\
H^{-5}(X_{13})=&11\ 13\\
&10\ 16\\
&4\ 14\\
&8\\
\end{array}
\end{displaymath}
\begin{displaymath}
\begin{array}{cc}
&\\
&13\\
H^{-4}(X_{13})=&10\\
&4\\
&\\
&\\
\end{array}
\
\begin{array}{cc}
&\\
&13\\
H^{-3}(X_{13})=&10\\
&4\\
&\\
&\\
\end{array}
\
\begin{array}{cc}
&\\
&\\
H^{-2}(X_{13})=&6\\
&\\
&\\
&\\
\end{array}
\
\begin{array}{cc}
&\\
&\\
H^{-1}(X_{13})=&6\\
&\\
&\\
&\\
\end{array}
\end{displaymath}
\begin{displaymath}
\begin{array}{cc}
&6\\
&15\\
H^{-5}(X_{14})=&11\ 13\\
&5\ 16\\
&3\ 12\\
&9\\
\end{array}
\end{displaymath}
\begin{displaymath}
\begin{array}{cc}
&\\
&11\\
H^{-4}(X_{14})=&5\\
&3\\
&\\
&\\
\end{array}
\
\begin{array}{cc}
&\\
&11\\
H^{-3}(X_{14})=&5\\
&3\\
&\\
&\\
\end{array}
\
\begin{array}{cc}
&\\
&\\
H^{-2}(X_{14})=&6\\
&\\
&\\
&\\
\end{array}
\
\begin{array}{cc}
&\\
&\\
H^{-1}(X_{14})=&6\\
&\\
&\\
&\\
\end{array}
\end{displaymath}
\begin{displaymath}
\begin{array}{cc}
&2\ 1\ 6\\
&15\\
H^{-6}(X_{15})=&11\ 13\\
&10\ 5\ 16\\
&14\ 12\\
\end{array}
\
\begin{array}{cc}
&1\ 2\ 6\\
&10\ 5\ 15\\
H^{-5}(X_{15})=&11\ 13\\
&10\ 5\\
&3\ 4\\
\end{array}
\
\begin{array}{cc}
&\\
&15\\
H^{-4}(X_{15})=&13\ 11\\
&10\ 5\\
&3\ 4\\
\end{array}
\end{displaymath}
\begin{displaymath}
\begin{array}{cc}
&6\ 15\\
&15\ 7\\
H^{-6}(X_{16})=&11\ 13\\
&5\ 10\ 16\\
&3\ 4\ 12\ 14\\
&9\ 8\\
\end{array}
\
\begin{array}{cc}
&\\
&\\
H^{-5}(X_{16})=&5\ 10\\
&3\ 4\\
&\\
&\\
\end{array}
\
\begin{array}{cc}
&\\
&\\
H^{-4}(X_{16})=&1\ 2\\
&5\ 10\\
&\\
&\\
\end{array}
\
\begin{array}{cc}
&\\
&\\
H^{-1}(X_{16})=&6\\
&\\
&\\
&\\
\end{array}
\end{displaymath}


\begin{thebibliography}{10}

\bibitem{a}
Al-Nofayee.
\newblock Derived equivalences for self-injective algebras and t-structures.
\newblock {\em Bristol PhD thesis}.

\bibitem{robbins}
D.Robbins.
\newblock Some results concerning brou\'e's abelian defect group conjecture.
\newblock {\em PhD Thesis}, 2008.

\bibitem{gap}
The~GAP group.
\newblock Gap - groups, algorithms and programming, version 4.2.
\newblock {\em (http://www-gap.dcs.st-and.ac.uk/)}.

\bibitem{jr89a}
J.Rickard.
\newblock Derived categories and stable equivalence.
\newblock {\em Journal of Pure and Applied Algebra \textbf{61} (303-317)},
  1989.

\bibitem{jr96}
J.Rickard.
\newblock Splendid equivalences : Derived categories and permutation modules.
\newblock {\em Proc. London Math. Soc.(3) \textbf{72} (331-358)}, 1996.

\bibitem{jr2002}
J.Rickard.
\newblock Equivalences of derived categories for symmetric algebras.
\newblock {\em Journal of Algebra \textbf{257} (460-481)}, 2002.

\bibitem{kkw02a}
S.Koshitani\~N.Kunugi\ K.Waki\.
\newblock Brou\'e's conjecture holds for principal $3$-blocks with elementary
  abelian defect group of order 9.
\newblock {\em J. Algebra \textbf{248} (575-604)}, 2002.

\bibitem{magma}
MAGMA.
\newblock http://www.maths.usyd.edu.au:8000/u/magma.

\bibitem{mh01}
M.Holloway.
\newblock Derived equivalences for group algebras.
\newblock {\em Bristol PhD thesis}, 2001.

\bibitem{rr98}
R.Rouquier.
\newblock The derived category of blocks with cyclic defect groups.
\newblock {\em in ``Derived equivalences for group rings'', Springer lecture
  notes in mathematics \textbf{1685} (199-220)}, 1998.

\bibitem{to97}
T.Okuyama.
\newblock Some examples of derived equivalent blocks of finite groups.
\newblock {\em preprint}, 1997.

\end{thebibliography}
\end{document}